\documentclass{amsart}
\usepackage{indentfirst,latexsym,bm,amsfonts,amsthm,amsmath,enumerate,graphicx,hyperref}
\usepackage{verbatim}
\DeclareMathOperator{\law}{law}
\DeclareMathOperator{\Id}{Id}
\DeclareMathOperator{\supp}{supp}

\newtheorem{theorem}{Theorem}[section]
\newtheorem{lemma}[theorem]{Lemma}
\newtheorem{corollary}[theorem]{Corollary}
\newtheorem{proposition}[theorem]{Proposition}
\theoremstyle{definition}
\newtheorem{definition}[theorem]{Definition}

\theoremstyle{remark}
\newtheorem{remark}[theorem]{Remark}

\numberwithin{equation}{section}

\begin{document}

\title{Busemann functions on the Wasserstein space}

\author{Guomin Zhu}
\address{Guomin Zhu, Department of Mathematics, Nanjing University, Nanjing 210093, China}
\email{zhu\_{}guomin@hotmail.com}

\author{Wen-Long Li}
\address{Wen-Long Li, School of Mathematics, Sun Yat-Sen University, Guangzhou, 510275, China}
\email{liwenlongchn@gmail.com}

\author{Xiaojun Cui}
\address{Xiaojun Cui, Department of Mathematics, Nanjing University, Nanjing 210093, China}
\email{xcui@nju.edu.cn}
\thanks{The third author is supported by the National Natural Science Foundation of China (Grants 11571166, 11631006, 11790272). The Project Funded by the Priority Academic Program Development of Jiangsu Higher Education Institutions (PAPD) and the Fundamental Research Funds for the Central Universities}


\date{\today}


\begin{abstract}
We study rays and co-rays in the Wasserstein space $P_p(\mathcal{X})$ ($p > 1$) whose ambient space $\mathcal{X}$ is a complete, separable, non-compact, locally compact length space.
We show that rays in the Wasserstein space can be represented as probability measures concentrated on the set of rays in the ambient space.
We show the existence of co-rays for any prescribed initial probability measure.
We introduce Busemann functions on the Wasserstein space and show that co-rays are negative gradient lines in some sense.
\end{abstract}

\maketitle

\section{Introduction}
The Wasserstein distance plays an important role in optimal transport theory \cite{Ambrosio, Santambrogio, Villani}
and shows its advantages in numerous topics such as mean-field games \cite{Carmona}, machine learning \cite{Arjovsky} and Hamilton-Jacobi equations \cite{AmbrosioFeng, Gangbo}.
Given an ambient space $\mathcal{X}$ which is a Polish space, namely complete separable metric space,
the Wasserstein space $P_p(\mathcal{X})$ consisting of Borel probability measures on $\mathcal{X}$ with finite $p$-moment is also Polish \cite[Theorem 6.18]{Villani}.
Besides, if $\mathcal{X}$ is a length (resp. geodesic) space, then $P_p(\mathcal{X})$ is also a length (resp. geodesic) space \cite{Lisini}.
Lisini \cite{Lisini} also characterized geodesics in the Wasserstein space as Borel probability measures on $C([0,T]; \mathcal{X})$ concentrated on the set of geodesics in $\mathcal{X}$.
In the setting of $\mathcal{X}$ to be an Hadamard space, Bertrand and Kloeckner \cite{Bertrand} investigated the geometry of $P_2(\mathcal{X})$, especially rays and geodesic boundary.

Busemann functions, introduced by Busemann \cite{Busemann}, are powerful tools for studying the topology and geometry of some kinds of non-compact spaces.
For example, Cheeger and Gromoll \cite{Cheeger} used them to prove the celebrated splitting theorem for manifolds of nonnegative Ricci curvature.
Bangert \cite{Bangert} investigated the dynamics on a Riemannian $2$-torus by means of Busemann functions on the covering space.
For more applications of Busemann functions, we refer to \cite{BangertEmmerich, BangertGutkin, Innami, Peterson, Shiohama, Sormani1, Sormani2}.
Furthermore the Busemann function has been introduced into the study of Lorentzian geometry \cite{Beem, Galloway, JinCui}.
Also, on a non-compact complete Riemannian manifold, they are viscosity solutions to the Eikonal equation (see \cite{CuiCheng}).

The aim of this paper is to extend Busemann functions to the Wasserstein space and thereby to get similar results as in the conventional case.
In a metric space $(\mathcal{Y},d)$, a ray is a curve $\gamma \in C(\mathbb{R}_+; \mathcal{Y})$ satisfying $d(\gamma_s,\gamma_t) = |s - t| d(\gamma_0,\gamma_1)$ for any $s,t \ge 0$ where $k_\gamma := d(\gamma_0,\gamma_1)$ is called the speed of $\gamma$.
Another ray $\tilde{\gamma}$ is said to be a co-ray from $p$ to $\gamma$ if it is the limit of a sequence $\{\zeta^n\}$ as $n \to \infty$,
where $\zeta^n$ is a geodesic connecting $p^n$ and $\gamma_{t_n}$ with $p^n \to p$ and $t_n \to +\infty$.
Usually a Busemann function is determined by a ray.
As a preparation we give a characterization of rays in the Wasserstein space, that is, every ray in $P_p(\mathcal{X})$ can be represented by a random ray in $\mathcal{X}$.
The main ingredient in the study of Busemann functions is co-ray.
A common method to show the existence of co-rays is selecting a sequence of geodesics with some good properties which by Ascoli's theorem admits a subsequence converging to a co-ray.
In the case of $P_p(\mathcal{X})$, the essential difficulty lies in the absence of local compactness:
any non-compact Wasserstein space can not be locally compact \cite[Remark 7.1.9]{Ambrosio}.
Meanwhile this fact implies that a Wasserstein space is not a G-space on which the Busemann function is initially defined,
because G-spaces are finitely compact i.e., every bounded infinity set has at least one accumulation point.

Without local compactness, our approach consists of two steps.
First we construct probability measures $\mathnormal{\Pi}^n$ on $C(\mathbb{R}_+; \mathcal{X})$ representing the geodesics $\zeta^n$.
The sequence $\{\mathnormal{\Pi}^n\}$ admits some accumulation points in the sense of weak convergence if it is confirmed to be tight.
Then we show that time sections of the weak convergent subsequences are actually uniformly integrable,
which means that the associated geodesics are convergent pointwisely with respect to the Wasserstein distance $W_p$.

With the assumptions of $p > 1$ and $\mathcal{X}$ to be a complete, separable, non-compact, locally compact length space, main results of this paper are stated as follows.

\begin{theorem}\label{main result 1}
Each point in $P_p(\mathcal{X})$ is the origin of at least one unit-speed ray.
\end{theorem}

$\mathcal{X}$ is called non-branching if any geodesic $\zeta: [a,b] \to \mathcal{X}$ is uniquely determined by its restriction to a nontrivial time-interval.
\begin{theorem}\label{main result 2}
Let $(\mu_t)_{t \ge 0}$ be a unit-speed ray in $P_p(\mathcal{X})$.
\begin{enumerate}[(i)]
\item  For any $\nu_0 \in P_p(\mathcal{X})$, there exists at least one co-ray from $\nu_0$ to $(\mu_t)_{t \ge 0}$;
\item  Moreover if $\mathcal{X}$ is non-branching, let $(\nu_t)_{t \ge 0}$ be one of the co-rays, then for any $\tau > 0$ the subray $(\nu_{t + \tau})_{t \ge 0}$ is the unique co-ray from $\nu_\tau$ to $(\mu_t)_{t \ge 0}$.
\end{enumerate}
\end{theorem}

For a unit-speed ray $(\mu_t)_{t \ge 0}$ in $P_p(\mathcal{X})$,
the Busemann function
$$b_\mu(\nu) = \lim\limits_{t \to +\infty} [W_p(\nu,\mu_t) - t]$$
is well defined on $P_p(\mathcal{X})$.

\begin{theorem}\label{main result 3}
Let $(\mu_t)_{t \ge 0}$ and $(\nu_t)_{t \ge 0}$ be two unit-speed rays in $P_p(\mathcal{X})$.
Assume $\mathcal{X}$ is non-branching, then $(\nu_t)_{t \ge 0}$ is a co-ray from $\nu_0$ to $(\mu_t)_{t \ge 0}$ if and only if
$$b_\mu(\nu_t) - b_\mu(\nu_s) = s - t$$
holds for any $s,t \ge 0$.
\end{theorem}

This paper is organized as follows.
Section 2 presents some preliminaries.
In Section 3 we study the structure of rays in the Wasserstein space.
More precisely, we show that each ray can be represented as a probability measure concentrated on the set of rays in the ambient space $\mathcal{X}$.
Section 4 presents the existence of co-rays in $P_p(\mathcal{X})$ and thus Theorem \ref{main result 1} and Theorem \ref{main result 2} (i) are proved.
To prove a sequence of geodesic is compact in the pointwise convergence topology with respect to $W_p$,
the key point is to show the uniformly integrability of the projections at any fixed time $\tau \ge 0$.
In Section 5 we define the Busemann function on the Wasserstein space and obtain some analogue fundamental properties to the conventional case.
The proofs of Theorem \ref{main result 2} (ii) and Theorem \ref{main result 3} are completed in this section.

\section{Preliminaries}

\subsection{Convergence in the Wasserstein space}

Let $(\mathcal{X},d)$ be a Polish space.
For $p \ge 1$, the Wasserstein space $P_p(\mathcal{X})$ of order $p$ is the set of Borel probability measures with finite $p$-moments, i.e.,
$$P_p(\mathcal{X})=\left\{\mu\in P(\mathcal{X})\left|\ \int_\mathcal{X} d(x_0,x)^p d\mu(x) < +\infty\right\}\right.,$$
where $x_0 \in \mathcal{X}$ is fixed.
This space does not depend on the choice of $x_0$.
Given $\mu$, $\nu \in P_p(\mathcal{X})$,
we denote by $\mathnormal{\Pi}(\mu, \nu)$ the set of Borel probability measure on $\mathcal{X} \times \mathcal{X}$ whose marginals are $\mu$ and $\nu$ respectively.
Elements in $\mathnormal{\Pi}(\mu, \nu)$ are called couplings of $(\mu, \nu)$.
The Wasserstein distance of order $p$ between $\mu$ and $\nu$ is defined by
\begin{equation}
W_p(\mu,\nu)  = \left(\min\limits_{\pi\in\Pi(\mu,\nu)}\int_\mathcal{X}d(x,y)^p d\pi(x,y)\right)^{1/p}.
\end{equation}
A coupling $\pi$ is said to be optimal if it achieves the minimum.

The space $C([0,T];\mathcal{X})$ of continuous curves in $\mathcal{X}$ equipped with the metric $\rho_{_T}$ is a Polish space, where
$$\rho_{_T}(\alpha, \beta) = \sup\limits_{0 \le t \le T}d(\alpha(t),\beta(t)), \text{~for~} \alpha, \beta\in C([0,T]; \mathcal{X}).$$
For $\alpha, \beta \in C(\mathbb{R}_+; \mathcal{X})$, we define
$$\rho(\alpha,\beta) = \sum\limits_{N\in\mathbb{N}}2^{-N}\frac{\rho_{_N}(\alpha,\beta)}{1 + \rho_{_N}(\alpha,\beta)}.$$
Then $(C(\mathbb{R}_+; \mathcal{X}), \rho)$ is also a Polish space \cite{Whitt}.

Let $P(\mathcal{X})$ denote the set of Borel probability measures on $\mathcal{X}$.
The support of $\mu \in P(\mathcal{X})$ defined by
\begin{equation}
\supp \mu = \{ x \in \mathcal{X}|\ \mu(B_r(x)) > 0 \text{~for~} r > 0\}
\end{equation}
is the smallest closed set on which $\mu$ is concentrated.
We say that a sequence $\{\mu^n\} \subset P(\mathcal{X})$ converges weakly to $\mu \in P(\mathcal{X})$, denoted by $\mu^n \Rightarrow \mu$, if
\begin{equation}
\lim\limits_{n \to \infty} \int_\mathcal{X} f(x) d\mu_n(x) = \int_\mathcal{X} f(x) d\mu(x)
\end{equation}
for every bounded continuous function $f$ on $\mathcal{X}$.

\begin{proposition}[{\cite[Proposition 5.1.8]{Ambrosio}}]\label{support convergence}
If $\{\mu^n\} \subset P(\mathcal{X})$ converges weakly to $\mu \in P(\mathcal{X})$, then for any $x \in \supp \mu$ there exists a sequence $x^n \in \supp \mu^n$ such that $\lim\limits_{n \to \infty} x^n = x$.
\end{proposition}

$S \subset P(\mathcal{X})$ is said to be tight if for any $\varepsilon > 0$,
there exists a compact set $K_\varepsilon \subset \mathcal{X}$ such that $\mu[K_\varepsilon] > 1 - \varepsilon$ for every $\mu \in P(\mathcal{X})$.
Every single probability measure on a Polish space is itself tight (see e.g. \cite[Theorem 1.3]{Billingsley}).
The following theorem indicates why the tightness makes sense.
\begin{theorem}[Prokhorov]\label{prokhorov}
Let $\mathcal{X}$ be a Polish space. $S \subset P(\mathcal{X})$ is tight if and only if it is relatively compact in $P(\mathcal{X})$.
\end{theorem}

There is a tightness criterion for a sequence of probability measures on the space $C(\mathbb{R}_+; \mathcal{X})$.

\begin{theorem}[{\cite[Theorem 4]{Whitt}}]\label{criterion}
Let $\{P^n\}$ be a sequence of probability measures on $C(\mathbb{R}_+; \mathcal{X})$.
Then $\{P^n\}$ is tight if and only if these two conditions hold:

(i) For each $t \ge 0$ and $\eta > 0$, there exists a compact set $K_t$ in $\mathcal{X}$ such that
$$P^n\{x \in C(\mathbb{R}_+; \mathcal{X})|\ x(t) \in K_t\} > 1 - \eta, \text{~for any~} n \ge 1.$$

(ii) For each $j \ge 1$ and $\varepsilon, \eta > 0$, there exists a $\delta \in (0,1)$, and an $n_0 \in \mathbb{N}$ such that
$$P^n\{x \in C(\mathbb{R}_+; \mathcal{X})|\ w_x^j(\delta) \ge \varepsilon\} \le \eta, \text{~for any~} n \ge n_0,$$
where $w_x^j(\delta) = \sup\limits_{\substack{0 \le s,t \le j\\|s-t| < \delta}}d(x(s), x(t))$.
\end{theorem}

The next theorem provides a more effective description of the convergence with respect to the Wasserstein distance.

\begin{theorem}[{\cite[Theorem 7.1.5]{Ambrosio}}]\label{wasserstein converge}
Let $(\mathcal{X},d)$ be a Polish space and $p \ge 1$.
Let $\{\mu^n\}_{n \in \mathbb{N}} \subset P_p(\mathcal{X})$ and $\mu \in P_p(\mathcal{X})$, then the following statements are equivalent:

(i) for some $x_0 \in \mathcal{X}$,
$\mu^n \Rightarrow \mu$ and
$$\lim\limits_{R \to \infty} \limsup\limits_{n \to \infty} \int_{d(x_0,x) \ge R} d(x_0,x)^p d\mu^n(x) = 0;$$

(ii) $W_p(\mu^n, \mu) \to 0$ as $n \to \infty$.

\end{theorem}

\subsection{Geodesics in the Wasserstein space}

Let $(\mathcal{Y},d)$ be a metric space.
The length of a continuous curve $\zeta: [a,b] \to \mathcal{Y}$ is defined by
\begin{equation}
L(\zeta) = \sup_{N \in \mathbb{N}} \sup_{a = t_0 < t_1 < \cdots < t_N = b} \sum_{i = 0}^{N - 1} d(\zeta_{t_i},\zeta_{t_{i+1}}).
\end{equation}
$(\mathcal{Y},d)$ is said to be a length space if for any $x, y \in \mathcal{Y}$,
\begin{equation}\label{length space}
d(x,y) = \inf\limits_{\zeta \in C([0,1];\mathcal{Y})} \{L(\zeta)|\ \zeta_0 = x, ~\zeta_1 = y\}.
\end{equation}
$\mathcal{Y}$ is a geodesic space if the infimum in equation (\ref{length space}) is attainable for any $x, y \in \mathcal{Y}$.
$\zeta$ is called a constant-speed minimizing geodesic segment if
\begin{equation}
d(\zeta_s,\zeta_t) = \frac{|t - s|}{b - a} d(\zeta_a,\zeta_b) \text{~for any~} s,t \in [a,b].
\end{equation}
For convenience, throughout this paper we use the single word ``geodesic'' instead.

The next statement is a straight corollary to \cite[Corollary 7.22]{Villani} via simple reparameterization.
This conclusion is twofold: The Wasserstein space over a complete separable locally compact length space is a geodesic space;
Geodesics in such a Wasserstein space can be considered as probability measures concentrated on the set of geodesics in the ambient space.

\begin{proposition}\label{existence}
Let $p > 1$ and let $(\mathcal{X}, d)$ be a complete separable, locally compact length space.
Given $\mu, \nu \in P_p(\mathcal{X})$, let $L = W_p(\mu, \nu)$.
Then for any continuous curve $(\mu_t)_{0 \le t \le L}$ in $P(\mathcal{X})$ with $\mu_0 = \mu$, $\mu_L = \nu$, the following properties are equivalent:

(i) $\mu_t$ is the law of $\zeta_t$, where $\zeta: [0, L] \to \mathcal{X}$ is a random geodesic such that $(\zeta_{_0}, \zeta_{_L})$ is an optimal coupling;

(ii) $(\mu_t)_{0 \le t \le L}$ is a unit-speed geodesic in the space $P_p(\mathcal{X})$.

Moreover, for any given $\mu, \nu \in P_p(\mathcal{X})$, there exists at least one such curve.
We denote by $T(\mu, \nu)$ the set of unit-speed geodesics from $\mu$ to $\nu$.
\end{proposition}

Let $f: \mathcal{Y}_1 \to \mathcal{Y}_2$ be a Borel map between Polish spaces, and $\lambda$ be a Borel measure on $\mathcal{Y}_1$.
The push-forward of $\lambda$, denoted by $f_\# \lambda$, is defined by $(f_\# \lambda)[A] = \lambda[f^{-1}(A)]$ for any Borel subset $A$.
In the case of Proposition \ref{existence}, let $\mathnormal{\Pi}$ be the law of $\zeta$, then $\mu_t = (e_t)_\# \mathnormal{\Pi}$ where $e_t: \zeta \mapsto \zeta_t$ be the canonical projection.
Besides, this work naturally prompts us to study rays in the Wasserstein space.

\section{Characterization of rays in the Wasserstein space}

In the rest of this paper, we always assume that $(\mathcal{X},d)$ is a complete, separable non-compact, locally compact length space and the order $p > 1$.
Recall that a ray in $\mathcal{X}$ is a curve $\gamma \in C(\mathbb{R}_+; \mathcal{X})$ satisfying
\begin{equation}
d(\gamma_s,\gamma_t) = |t - s| d(\gamma_0,\gamma_1) \text{~for any~} s,t \ge 0,
\end{equation}
where $k_\gamma = d(\gamma_0,\gamma_1)$ is called the speed of $\gamma$.

\begin{lemma}
Let $\Gamma$ be the set of rays in $\mathcal{X}$, then $\Gamma$ is closed in $(C(\mathbb{R}_+; \mathcal{X}),\rho)$.
As a consequence, $(\Gamma, \rho)$ is a Polish space.
\end{lemma}

\begin{proof}
Let $\{\gamma^n\}$ be a Cauchy sequence of $\Gamma$, then there exists a $\gamma: \mathbb{R}_+ \to \mathcal{X}$ such that $\rho(\gamma^n,\gamma)\to 0$.
Denote $k = \lim\limits_{n \to \infty} d(\gamma^n_0, \gamma^n_1)$.
The limit exists since $\lim\limits_{n \to \infty} d(\gamma^n_t,\gamma_t) = 0$ for any $t \ge 0$.
For any $t_1, t_2 \ge 0$,
\begin{equation}\label{ieqna}
|d(\gamma^n_{t_1},\gamma^n_{t_2}) - k|t_1 - t_2|| = |t_1 - t_2||d(\gamma^n_0,\gamma^n_1) - k| \to 0.
\end{equation}
By the triangle inequality,
\begin{eqnarray}\label{ieqnb}
|d(\gamma_{t_1},\gamma_{t_2}) - d(\gamma^n_{t_1},\gamma^n_{t_2})| & = & |d(\gamma_{t_1},\gamma_{t_2}) - d(\gamma^n_{t_1},\gamma^n_{t_2})|\nonumber\\
& \le & d(\gamma^n_{t_1},\gamma_{t_1}) + d(\gamma^n_{t_2},\gamma_{t_2}) \to 0.
\end{eqnarray}
Combining (\ref{ieqna}) and (\ref{ieqnb}), we have
$|d(\gamma_{t_1},\gamma_{t_2}) - k|t_1 - t_2|| = 0$, then the conclusion follows.
\end{proof}

As shown in Proposition \ref{existence}, if geodesics in $P_p(\mathcal{X})$ do not share the same lengths,
then their corresponding random curves are defined on different time intervals.
So we introduce the mapping $E_T: C([0,T]; \mathcal{X}) \to C(\mathbb{R}_+; \mathcal{X})$ by
$$(E_T(\zeta))_s = \zeta_{\min\{s, T\}}, s \ge 0$$
in order to extend their sample paths onto the common space $C(\mathbb{R}_+; \mathcal{X})$.

\begin{definition}\label{extend}
Let $(\mu_t)_{0 \le t \le L}$ be a geodesic in $P_p(\mathcal{X})$ and $\pi$ be a probability measure on $C([0,L];\mathcal{X})$ such that $\mu_t = (e_t)_\# \pi$ for $0 \le t \le L$.
The probability measure $(E_L)_\# \pi$ on $C(\mathbb{R}_+; \mathcal{X})$ is called a lifting of the geodesic.
\end{definition}
It can be seen from Proposition \ref{existence} that each geodesic in $P_p(\mathcal{X})$ admits a lifting.

\begin{lemma}\label{tail ieqn}
For $\mu$, $\nu \in P_p(\mathcal{X})$, let $\pi$ be an optimal coupling.
Then for $R > 0$,
\begin{equation}\label{ieqnc}
\pi \{(x,y)|\ d(x,y)>R\} \le \left[\frac{W_p(\mu,\nu)}{R}\right]^p.
\end{equation}
\end{lemma}

\begin{proof}
Given $R > 0$, by the definition of $W_p$,
\begin{eqnarray*}
\int_{d(x,y)>R} R^p d\pi(x,y) & \le & \int_{d(x,y)>R}d(x,y)^p d\pi(x,y)\\
& \le & \int_{\mathcal{X}\times\mathcal{X}} d(x,y)^p d\pi(x,y)\\
& = & W_p^p(\mu,\nu).
\end{eqnarray*}
So we obtain the inequality (\ref{ieqnc}).
\end{proof}

\begin{theorem}\label{tight}
Let $\{\nu^n\}$ and $\{\mu^n\}$ be sequences in $P_p(\mathcal{X})$ such that $\{\nu^n\}$ is tight and $L_n := W_p(\mu^n,\nu^n) \to +\infty$.
For each $n$, let $\mathnormal{\Pi}^n$ be a lifting of an element in $T(\nu^n,\mu^n)$, then $\{\mathnormal{\Pi}^n\}$ is tight.
\end{theorem}

\begin{proof}
Let $\nu^n_t = (e_t)_\# \mathnormal{\Pi}^n$ for $t \le L_n$.
Fix an arbitrary $t > 0$.

(i) Given any $\eta > 0$, there exists an $N$ such that $L_n \ge t$ for $n > N$.
Recall that each single probability measure on a Polish space is tight.
For $n \le N$, there exists a compact set $K^n_t$ such that
\begin{equation}
\mathnormal{\Pi}^n\{\gamma^n \in C(\mathbb{R}_+; \mathcal{X})|\ \gamma^n_t \in K^n_t\} = \nu_t^n[K^n_t] > 1 - \frac{\eta}{2N}.
\end{equation}
While for $n > N$, by the tightness of $\{\nu^n_0\}$ there exists a constant $D_0$ such that
\begin{equation}\label{ieqn eta}
\mathnormal{\Pi}^n \{\gamma^n|\ d(\gamma^n_0,x_0) > D_0\} = \nu^n_0 \{x|\ d(x,x_0) > D_0\} < \frac{\eta}{4}.
\end{equation}
By Proposition \ref{existence}, for each $n \in \mathbb{N}$, $\pi^n := (e_0, e_{L_n})_\# \mathnormal{\Pi}^n$ is an optimal coupling of $(\nu^n, \mu^n)$.
By Lemma \ref{tail ieqn},
\begin{equation}\label{ieqn t}
\mathnormal{\Pi}^n \{\gamma^n|\ d(\gamma^n_0,\gamma^n_t) > R\}  = \pi^n \{(x,y)|\ d(x,y) > R\} \le \left(\frac{t}{R}\right)^p.
\end{equation}
For $D_t \ge t (4/\eta)^{1/p} + D_0$, from (\ref{ieqn eta}) and (\ref{ieqn t}) we have
\begin{eqnarray*}
&&\mathnormal{\Pi}^n \{\gamma^n|\ d(\gamma^n_t,x_0)>D_t\}\\
& \le & \mathnormal{\Pi}^n \{\gamma^n|\ d(\gamma^n_0,x_0)>D_0\} + \mathnormal{\Pi}^n \{\gamma^n|\ d(\gamma^n_0,x_0) \le D_0, d(\gamma^n_0,\gamma^n_t) > D_t - D_0\}\\
& \le & \mathnormal{\Pi}^n \{\gamma^n|\ d(\gamma^n_0,x_0)>D_0\} + \mathnormal{\Pi}^n \{\gamma^n|\ d(\gamma^n_0,\gamma^n_t) > D_t-D_0\}\\
& \le & \frac{\eta}{4} + \left(\frac{t}{D_t-D_0}\right)^p\\
& \le & \eta/2,
\end{eqnarray*}
which means there is a compact set $K_t = \{x|\ d(x,x_0) \le D_t\} \cup \bigcap\limits_{i = 1}^N K^i_t$ such that
\begin{equation}
\mathnormal{\Pi}^n \{\gamma^n|\ \gamma^n \in K_t\} > 1 - \eta.
\end{equation}

(ii) For any fixed $j \ge 1$, by Definition \ref{extend},
for $n > j$ the curves $(\nu^n_t)_{0 \le t \le j}$ are geodesics in $P_p(\mathcal{X})$.
For any fixed $\varepsilon, \eta > 0$, let $\delta < \varepsilon \eta^{1/p}$, then
\begin{equation}
\mathnormal{\Pi}^n \{\gamma^n|\ w^j_{\gamma^n}(\delta) > \varepsilon \} = \mathnormal{\Pi}^n \left\{\gamma^n \left|\ \sup\limits_{\substack{0 \le s,t \le j\\
|s-t|<\delta}}d(\gamma^n_s,\gamma^n_t) > \varepsilon \right. \right\} \le \left(\frac{\delta}{\varepsilon}\right)^p < \eta
\end{equation}
by using Lemma \ref{tail ieqn} again.

Hence the tightness of $\{\mathnormal{\Pi}^n\}$ follows from Theorem \ref{criterion}.
\end{proof}

Now we are able to characterize rays in $P_p(\mathcal{X})$. 
A similar representation can be found in \cite[Proposition 3.2]{Bertrand} for the 2-Wasserstein space over an Hadamard space.

\begin{corollary}\label{random ray}
The following statements are equivalent:

(i) $(\mu_t)_{t \ge 0}$ is a ray in $P_p(\mathcal{X})$;

(ii) For any $t_1,t_2$ with $0 \le t_1 < t_2, (\mu_t)_{t_1 \le t \le t_2}$ is a geodesic;

(iii) $\mu_t$ is the law of $\gamma_t$,
where $\gamma$ is a random ray such that $(\gamma_{t_1},\gamma_{t_2})$ is an optimal coupling for any $t_1, t_2$ with $0 \le t_1 < t_2$.
\end{corollary}

\begin{proof}
(i) and (ii) are equivalent by the definition of ray and (iii)$\Rightarrow$(ii) is obvious.

(ii)$\Rightarrow$(iii).
For $j \in \mathbb{N}$, let $\mathnormal{\Pi}^{j}$ be a lifting of $(\mu_t)_{0\le t \le j}$ and
$$\Gamma^j = \{\zeta \in C([0,j]; \mathcal{X})|\ \zeta \text{~is a geodesic}\}.$$
By Proposition \ref{existence}, $\mathnormal{\Pi}^{j}$ is concentrated on $E_j(\Gamma^j)$ such that $(e_0,e_j)_\# \mathnormal{\Pi}^{j}$ is an optimal coupling.
By Theorem \ref{tight} and Theorem \ref{prokhorov}, it admits a subsequence $\{\mathnormal{\Pi}^{j'}\}$ which converges weakly to some measure $\mathnormal{\Pi}$ on $C(\mathbb{R}_+; \mathcal{X})$.
For any $\gamma \in \supp \mathnormal{\Pi}$, by Proposition \ref{support convergence}, there exists $\zeta^{j'} \in E_{j'}(\Gamma^{j'})$ with
\begin{equation}\label{zeta to gamma}
\lim\limits_{j' \to \infty} \rho(\zeta^{j'},\gamma) = 0.
\end{equation}
For any $0 \le s_1 < s_2$, choose $T \ge \max\{1, s_2\}$, then (\ref{zeta to gamma}) implies
\begin{equation}\label{zeta converge}
d(\zeta^{j'}_s, \gamma_s) \to 0 \text{~for any~} s \in [0,T].
\end{equation}
Notice that $d(\zeta^{j'}_{s_1}, \zeta^{j'}_{s_2}) = (s_2 - s_1) d(\zeta^{j'}_0, \zeta^{j'}_1)$ when $j' > T$.
(\ref{zeta converge}) yields
\begin{equation}
d(\gamma_{s_1},\gamma_{s_2}) = (s_2 - s_1) d(\gamma_0,\gamma_1).
\end{equation}
It follows that $\gamma \in \Gamma$, and consequently $\mathnormal{\Pi}$ is the law of a random ray.

For $0 \le t_1 < t_2$, denote $\pi_{t_1,t_2}^{j'} = (e_{t_1},e_{t_2})_\# \mathnormal{\Pi}^{j'}, \pi_{t_1,t_2} = (e_{t_1},e_{t_2})_\# \mathnormal{\Pi}$,
then $\pi^{j'}_{t_1,t_2} \Rightarrow \pi_{t_1,t_2}$.
By the lower semicontinuity of the map $\pi \mapsto \int d^p d\pi$,
\begin{eqnarray*}
W_p^p(\mu_{t_1},\mu_{t_2}) & \le & \int_{\mathcal{X} \times \mathcal{X}} d^p(x,y) d\pi_{t_1,t_2}\\
& \le & \liminf\limits_{j' \to \infty} \int_{\mathcal{X} \times \mathcal{X}} d^p(x,y) d\pi_{t_1,t_2}^{j'}\\
& = & W_p^p(\mu_{t_1},\mu_{t_2}),
\end{eqnarray*}
which means $\pi_{t_1,t_2}$ is an optimal coupling.
\end{proof}

\section{Existence of co-rays in the Wasserstein space}
In the conventional case, co-rays play a central role in the study of Busemann functions \cite{Busemann}.
This notion also make sense in the present case.
The existence of co-rays in the Wasserstein space will be proved in this section.

\begin{definition}[Co-ray]
Let $(\mu_t)_{t \ge 0}$ be a unit-speed ray in $P_p(\mathcal{X})$.
We say another ray $(\nu_t)_{t \ge 0}$ is a co-ray from $\nu_0$ to $(\mu_t)_{t \ge 0}$,
if there exist:
\begin{itemize}
  \item  $\{t_n\} \subset \mathbb{R}_+$ tends to infinity,
  \item  $\{\nu_0^n\} \subset P_p(\mathcal{X})$ with $W_p(\nu^n_0,\nu_0) \to 0$,
  \item  for $n \in \mathbb{N}$, $(\nu^n_t)_{0 \le t \le L_n} \in T(\nu_0^n, \mu_{t_n})$ where $L_n = W_p(\nu^n_0, \mu_{t_n})$
\end{itemize}
such that $\lim\limits_{n \to \infty} W_p(\nu_t^n, \nu_t) = 0$ for every $t \ge 0$.
\end{definition}

The gluing lemma (see e.g. \cite[Lemma 5.3.2]{Ambrosio}) is often used in optimal transport to connect two couplings.
However this instrument can not meet our demand while two random curves get involved, so we cite a theorem here in order to obtain another version of gluing lemma.

\begin{theorem}[{\cite[Theorem A.1]{Acosta}}]\label{canonical projections}
Let $J$ be an arbitrary index set.
For each $j \in J$, let $S_j, T_j$ be Polish spaces, $S = \prod\limits_{j \in J} S_j, T = \prod\limits_{j \in J} T_j$.
Also, let $p_j: S \to S_j, q_j: T \to T_j$ be canonical projections, $\phi_j: S_j \to T_j$ be a measurable map.
If $\mu_j$ is a probability measure on $S_j$ and $\lambda$ is a probability measure on $T$ such that $(\phi_j)_\# \mu_j = (q_j)_\# \lambda$, $\forall j \in J$.
Then there exists a probability measure $\sigma$ on $S$ such that:

(i) $p_j \sigma = \mu_j$, for all $j \in J$;

(ii) $((\phi_j \circ p_j)_{j \in J})_\# \sigma = \lambda$.
\end{theorem}

The modified gluing lemma allow us to glue a random curve and a coupling together.
\begin{lemma}\label{modified gluing lemma}
Let $\mathcal{Y}$ be a Polish space, $\alpha \in P(C([0,T];\mathcal{X})), \beta \in P(\mathcal{X} \times \mathcal{Y})$ and $\pi^i$ be the natural projection onto the $i$-th coordinate, $i = 1,2$.
If $(e_T)_\# \alpha = (\pi^1)_\# \beta$,
then there exists a $\delta \in P(C([0,T];\mathcal{X}) \times \mathcal{Y})$ such that
$$(\pi^1)_\# \delta = \alpha, ~(e_T \pi^1, \pi^2)_\# \delta = \beta.$$
\end{lemma}

\begin{proof}
Let $S_1 = C([0,T];\mathcal{X}), T_1 = \mathcal{X}, \phi_1 = e_T$;
$S_2 = \mathcal{Y} = T_2, \phi_2 = \Id_{\mathcal{Y}}$;
$\mu_1 = \alpha, \mu_2 = (\pi^2)_\# \beta = (q_2)_\# \beta, \lambda = \beta$.
Then
$$(\phi_1)_\# \mu_1 = (e_T)_\# \alpha = (\pi^1)_\# \beta = (q_1)_\# \beta;$$
$$(\phi_1)_\# \mu_2 = (\Id_{\mathcal{Y}})_\# ((q_2)_\# \beta) = (q_2)_\# \beta.$$

Applying Theorem \ref{canonical projections} for $J = \{1,2\}$,
there exists a probability measure $\delta$ on $S_1 \times S_2 = C([0,T];\mathcal{X}) \times \mathcal{Y}$ such that
$$(\pi^1)_\# \delta = (p_1)_\# \delta = \mu_1 = \alpha,$$
$$(e_T \pi^1, \pi^2)_\# \delta = (\phi_1 \circ p_1, \phi_1 \circ p_2)_\# = \lambda = \beta.$$
\end{proof}

\begin{theorem}\label{coray existence}
Let $(\mu_t)_{t \ge 0}$ be a unit-speed ray in $P_p(\mathcal{X})$, $p > 1$.
Given an arbitrary $\nu_0 \in P_p(\mathcal{X})$, for any:
\begin{itemize}
  \item $\{t_n\}$ increasing to infinity,
  \item $\{\nu^n_0\} \subset P_p(\mathcal{X})$ with $W_p(\nu^n,\nu_0) \to 0$,
  \item $(\nu^n_t)_{0 \le t \le L_n} \in T(\nu^n_0, \mu_{t_n})$, $n \in \mathbb{N}$ where $L_n = W_p(\nu^n_0,\mu_{t_n})$,
\end{itemize}
there exists a subsequence of $\{(\nu^n_t)_{0 \le t \le L_n}\}$ which converges to a co-ray from $\nu_0$ to $(\mu_t)_{t \ge 0}$.
\end{theorem}

\begin{proof}
Let $\Lambda^n$ be a lifting of $(\nu_t^n)_{0 \le t \le L_n}$.
By Theorem \ref{tight}, there exists a subsequence of $\{\Lambda^n\}$, still denoted by the same notation,
which converges weakly to a probability measure $\Lambda$ on $C(\mathcal{X})$.
Then for arbitrary fixed $\tau \ge 0$,
$$\nu_\tau^n = (e_\tau)_\# \Lambda^n \Rightarrow (e_\tau)_\# \Lambda := \nu_\tau.$$
To show that $W_p(\nu^n_\tau,\nu_\tau) \to 0$ as $n \to \infty$, by Theorem \ref{wasserstein converge}, it remains to prove that
\begin{equation}\label{uniform integrable}
\lim\limits_{R \to \infty} \limsup\limits_{n \to \infty} \int_{d(x_0,z) \ge R} d(x_0,z)^p d\nu^n_\tau(z) = 0.
\end{equation}
By Proposition \ref{existence}, for each $n$, there exists $\alpha^n \in P(C([0,L_n];\mathcal{X}))$ satisfying
$$(e_t)_\# \alpha^n = \nu_t^n \text{~for~} t \in [0,L_n].$$
By Corollary \ref{random ray}, there is a random ray $(\eta_t)_{t \ge 0}$ such that $(\eta_0,\eta_t)$ are optimal couplings and $\mu_t = \law(\eta_t)$ for any $t \ge 0$.
We write $\beta$ for the distribution of $\eta$.
For each $n$, let $\beta^n = (e_0, e_{t_n})_\# \beta$, then it is an optimal coupling of $(\mu_0,\mu_{t_n})$.
By Lemma \ref{modified gluing lemma},
we can construct a sequence of probability measures $\{\mathnormal{\Pi}^n\} \subset P(C([0,L_n];\mathcal{X}) \times \mathcal{X})$ satisfying
\begin{equation}
(\pi^1)_\# \mathnormal{\Pi}^n = \alpha^n,~(e_{L_n} \pi^1, \pi^2)_\# \mathnormal{\Pi}^n = \beta^n.
\end{equation}
By the triangle inequality,
\begin{eqnarray*}
|L_n - t_n| & = & |W_p(\nu_0^n, \mu_{t_n}) - W_p(\mu_0, \mu_{t_n})|\\
& \le & W_p(\nu_0^n, \mu_0)\\
& \le & W_p(\nu_0^n, \nu_0) + W_p(\nu_0, \mu_0).
\end{eqnarray*}
Then
\begin{equation}\label{equal speed}
\lim\limits_{n \to \infty} \frac{t_n}{L_n} = 1.
\end{equation}

There exists an $N$ such that $L_n > \tau$ for all $n > N$.
In this case, for $\mathnormal{\Pi}^n$-a.e. $(\gamma^n, y) \in C([0,L_n];\mathcal{X}) \times \mathcal{X}$, by the triangle inequality,
\begin{eqnarray*}
d(x_0,\gamma^n_\tau) & \le & d(x_0, \gamma^n_0) + d(\gamma^n_0, \gamma^n_\tau)\\
& = & d(x_0,\gamma^n_0) + \tau/{L_n} d(\gamma^n_0,\gamma^n_{L_n})\\
& \le & d(x_0,\gamma^n_0) + \tau/{L_n} [d(x_0,\gamma^n_0) + d(x_0,y) + d(y,\gamma^n_{L_n})]\\
& = & \left(1 + \frac{\tau}{L_n}\right) d(x_0,\gamma^n_0) + \frac{\tau}{L_n} d(x_0,y) + \frac{\tau}{L_n} d(y,\gamma^n_{L_n}).
\end{eqnarray*}
The first equality follows from that each $\gamma^n$ is a random geodesic.

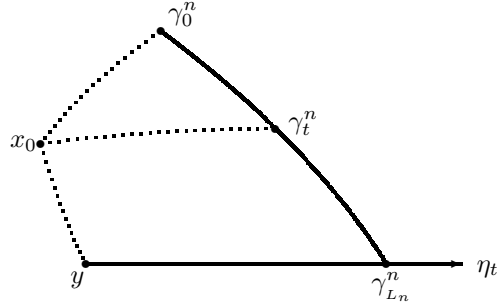
\begin{center}
\setlength{\unitlength}{1mm}
\begin{picture}(60,38)
\linethickness{1pt}
\put(10,4){\vector(1,0){50}}   
\put(10,4){\circle*{1}}        
\put(8,1.5){$y$}
\put(62,3){$\eta_t$}
\put(50,4){\circle*{1}}        
\put(48,0.6){$\gamma^n_{_{L_n}}$}
\qbezier(20,35)(40,20)(50,4)   
\put(20,35){\circle*{1}}       
\put(21,36.5){$\gamma^n_0$}
\put(35.2,22){\circle*{1}}     
\put(37,22){$\gamma^n_t$}
\put(4,20){\circle*{1}}        
\put(0,19.5){$x_0$}
\bezier{20}(4,20)(10,28)(20,35)
\bezier{26}(4,20)(20,22)(35.2,22)
\bezier{16}(4,20)(7,10)(10,4)
\end{picture}
\end{center}

Applying Jensen's inequality we have
\begin{eqnarray*}
& & \int_{d(x_0,z) \ge R} d^p(x_0,z) d\nu_\tau^n(z)\\
& = & \int_{d(x_0,\gamma_\tau^n) \ge R} d^p(x_0,\gamma_\tau^n) d\mathnormal{\Pi}^n(\gamma^n,y)\\
& \le & 3^p \left(\frac{\tau}{L_n} + 1\right)^p \int_{d(x_0,\gamma_\tau^n) \ge R} d^p(x_0,\gamma_0^n) d\mathnormal{\Pi}^n(\gamma^n,y)\\
& & + \left(\frac{3\tau}{L_n}\right)^p \int_{d(x_0,\gamma_\tau^n) \ge R} d^p(x_0,y) d\mathnormal{\Pi}^n(\gamma^n,y)\\
& & + \left(\frac{3\tau}{L_n}\right)^p \int_{d(x_0,\gamma_\tau^n) \ge R} d^p(y,\gamma^n_{L_n}) d\mathnormal{\Pi}^n(\gamma^n,y)\\
& \le & 3^p \left(\frac{\tau}{L_n} + 1\right)^p \int_{d(x_0,\gamma_\tau^n) \ge R} d^p(x_0,\gamma_0^n) d\mathnormal{\Pi}^n(\gamma^n,y)\\
& & + \left(\frac{3\tau}{L_n}\right)^p W_p^p(\delta_{x_0},\mu_0)
+ \left(\frac{3\tau}{L_n}\right)^p \int_{d(x_0,\gamma_\tau^n) \ge R} d^p(y,\gamma^n_{L_n}) d\mathnormal{\Pi}^n(\gamma^n,y)\\
& := & 3^p \left(1 + \frac{\tau}{L_n}\right)^p I_1 + \left(\frac{3\tau}{L_n}\right)^p W_p^p(\delta_{x_0},\mu_0) + \left(\frac{3\tau}{L_n}\right)^p I_2.
\end{eqnarray*}
We will consider the three terms above respectively.

(i) Let $M = \sup\limits_n \{\frac{\tau}{L_n}\}$.
For arbitrary $\varepsilon > 0$, since $W_p(\nu_0^n,\nu_0) \to 0$, by Theorem \ref{wasserstein converge} there exists an $R_1$ such that
\begin{equation}
\limsup\limits_{n \to \infty} \int_{d(x_0,z) \ge R_1} d^p(x_0,z) d\nu_0^n(z) < \frac{\varepsilon}{4 \cdot 3^p (1+M)^p}.
\end{equation}
We can further find an $R_2 > R_1$ such that when $R > R_2$,
\begin{equation}
\left(\frac{R_1 \tau}{R - R_1}\right)^p < \frac{\varepsilon}{4 \cdot 3^p (1+M)^p}.
\end{equation}

In this case we estimate $I_1$ as following:
\begin{eqnarray*}
I_1 & = & \int_{d(x_0,\gamma_\tau^n) \ge R} d^p(x_0,\gamma_0^n) d\mathnormal{\Pi}^n(\gamma^n,y)\\
& = & \int_{d(x_0,\gamma_0^n) \ge R_1,\ d(x_0,\gamma_\tau^n) \ge R} d^p(x_0,\gamma_0^n) d\mathnormal{\Pi}^n(\gamma^n,y)\\
& & + \int_{d(x_0,\gamma_0^n) < R_1,\ d(x_0,\gamma_\tau^n) \ge R} d^p(x_0,\gamma_0^n) d\mathnormal{\Pi}^n(\gamma^n,y)\\
& \le & \int_{d(x_0,\gamma_0^n) \ge R_1} d^p(x_0,\gamma_0^n) d\mathnormal{\Pi}^n(\gamma^n,y)
+ R_1^p \cdot \alpha^n \{\gamma^n|\ d(\gamma_0^n,\gamma_\tau^n) > R - R_1\}\\
& \le & \int_{d(x_0,z) \ge R_1} d^p(x_0,z) d\nu_0^n(z) + \left(\frac{R_1 \tau}{R - R_1}\right)^p.
\end{eqnarray*}
The last inequality is obtained from Lemma \ref{tail ieqn}.
Consequently,
\begin{equation}
\limsup\limits_{n \to \infty} 3^p \left(1 + \frac{\tau}{L_n}\right)^p I_1
\le \limsup\limits_{n \to \infty} 3^p \left(1 + \frac{\tau}{L_n}\right)^p \frac{\varepsilon}{2 \cdot 3^p (1 + M)^p} < \frac{\varepsilon}{2}.
\end{equation}

(ii) Since $W_p(\mu_0,\mu_1) = 1$, there exists an $R_3$ such that
\begin{equation}\label{I2ieqn1}
\int_{d(\eta_0,\eta_1) \ge R_3} d^p(\eta_0,\eta_1) d\beta(\eta) < \frac{\varepsilon}{6 \cdot (3\tau)^p}.
\end{equation}
From the construction of $\mathnormal{\Pi}^n$, we have
\begin{eqnarray*}
& & \int_{d(y,\gamma^n_{L_n}) \ge t_n R_3} d^p(y,\gamma^n_{L_n}) d\mathnormal{\Pi}^n(\gamma^n,y)\\
& = & \int_{d(\eta_0,\eta_{t_n}) \ge t_n R_3} d^p(\eta_0,\eta_{t_n}) d\beta^n(\eta_0,\eta_{t_n})\\
& = & t^p_n \int_{d(\eta_0,\eta_1) \ge R_3} d^p(\eta_0,\eta_1) d\beta(\eta).
\end{eqnarray*}
Due to the tightness of $\{\nu_0^n\}$, there is an $R_4 > R_3$ with
\begin{equation}\label{I2ieqn2}
\nu_0^n\{z|\ d(x_0,z) \ge R_4\} < \frac{\varepsilon}{6 \cdot (3\tau R_3)^p} \text{~for all~} n > N.
\end{equation}
Applying the inequality (\ref{ieqnc}) again, we can obtain another $R_5 > R_4$ such that when $R > R_5$,
\begin{equation}\label{I2ieqn3}
\alpha^n \{\gamma^n|\ d(\gamma_0^n,\gamma_\tau^n) > R - R_4\} < \frac{\varepsilon}{6 \cdot (3\tau R_3)^p} \text{~for all~} n > N.
\end{equation}
In this case,
\begin{eqnarray*}
I_2 & = & \int_{d(x_0,\gamma_\tau^n) \ge R} d^p(y, \gamma_{L_n}^n) d\mathnormal{\Pi}^n(\gamma^n,y)\\
& = & \int_{d(x_0,\gamma_\tau^n) \ge R,\ d(y,\gamma_{L_n}^n) \ge t_n R_3} d^p(y,\gamma_{L_n}^n) d\mathnormal{\Pi}^n(\gamma^n,y)\\
& & + \int_{d(x_0,\gamma_\tau^n) \ge R,\ d(y,\gamma_{L_n}^n) < t_n R_3,\ d(x_0,\gamma_0^n) \ge R_4} d^p(y,\gamma_{L_n}^n) d\mathnormal{\Pi}^n(\gamma^n,y)\\
& & + \int_{d(x_0,\gamma_\tau^n) \ge R,\ d(y,\gamma_{L_n}^n) < t_n R_3,\ d(x_0,\gamma_0^n) < R_4} d^p(y,\gamma_{L_n}^n) d\mathnormal{\Pi}^n(\gamma^n,y)\\
& \le & \int_{d(y,\gamma_{L_n}^n) \ge t_n R_3} d^p(y,\gamma_{L_n}^n) d\mathnormal{\Pi}^n(\gamma^n,y)\\
& & + (t_n R_3)^p \mathnormal{\Pi}^n \{(\gamma^n,y)|\ d(x_0,\gamma_0^n) \ge R_4\}\\
& & + (t_n R_3)^p \mathnormal{\Pi}^n \{(\gamma^n,y)|\ d(\gamma_0^n,\gamma_\tau^n) > R - R_4\}\\
& = & t_n^p \int_{d(\eta_0,\eta_1) \ge R_3} d^p(\eta_0,\eta_1) d\beta(\eta)\\
& & + (t_n R_3)^p \nu_0^n \{z|\ d(x_0,z) \ge R_4\}\\
& & + (t_n R_3)^p \alpha^n \{\gamma^n|\ d(\gamma_0^n,\gamma_\tau^n) > R - R_4\}.
\end{eqnarray*}
Combining this inequality with (\ref{I2ieqn1}), (\ref{I2ieqn2}), (\ref{I2ieqn3}) and (\ref{equal speed}), we have
\begin{equation}
\limsup\limits_{n \to \infty} \left(\frac{3\tau}{L_n}\right)^p I_2 \le \frac{\varepsilon}{2}.
\end{equation}

(iii) Notice that $\lim\limits_{n \to \infty} (\frac{3\tau}{L_n})^p W_p^p(\delta_{x_0},\mu_0) = 0$. We obtain
\begin{equation}
\limsup\limits_{n \to \infty} \int_{d(x_0,z) \ge R} d^p(x_0,z) d\nu_\tau^n(z) < \varepsilon \text{~when~} R > \max\{R_2, R_5\}.
\end{equation}
We conclude from (\ref{uniform integrable}) that $\lim\limits_{n \to \infty} W_p(\nu_t^n, \nu_t) = 0$ for arbitrary $t \ge 0$, thus $\nu_t \in P_p(\mathcal{X})$.
Moreover, for $t,s \ge 0$,
$$W_p(\nu_t,\nu_s) = \lim\limits_{n \to \infty} W_p(\nu_t^n, \nu_s^n) = |t - s|,$$
which means $(\nu_t)_{t \ge 0}$ is also a unit-speed ray in $P_p(\mathcal{X})$.
Therefore it is a co-ray from $\nu_0$ to $(\mu_t)_{t \ge 0}$.
\end{proof}

The first part of Theorem \ref{main result 2} follows directly from Theorem \ref{coray existence}.

A complete and locally compact length space, by the Hopf-Rinow Theorem (see \cite[Theorem 2.5.28]{Burago}), is boundedly compact i.e., every closed metric ball is compact.
Then each point in $\mathcal{X}$ is the origin of some unit-speed rays, which is due to \cite[Proposition 10.1.1]{Papadopoulos}.
It is easily seen that the mapping $x \mapsto \delta_x$ is an isometric embedding from $\mathcal{X}$ to $P_p(\mathcal{X})$.
Let $\gamma$ be a unit-speed ray in $\mathcal{X}$ and $\mu_t = \delta_{\gamma_t}$, then $(\mu_t)_{t \ge 0}$ is a unit-speed ray in $P_p(\mathcal{X})$ accordingly.
Notice that for each $\nu_0 \in P_p(\mathcal{X})$, co-rays from $\nu_0$ to $(\mu_t)_{t \ge 0}$ are some of the unit-speed rays with the origin $\nu_0$, thus Theorem \ref{main result 1} holds.

\section{Busemann functions on the Wasserstein space}
\begin{definition}
Let $(\mu_t)_{t \ge 0}$ be a unit-speed ray in $P_p(\mathcal{X})$.
The Busemann function associated with $(\mu_t)_{t \ge 0}$ is defined by
\begin{equation}\label{busemann function}
b_\mu(\nu) = \lim\limits_{t \to +\infty} [W_p(\nu,\mu_t) - t].
\end{equation}
\end{definition}

\begin{remark}
To show that the limit in (\ref{busemann function}) exists and is finite, by triangle inequality,
$$W_p(\nu,\mu_0) \ge W_p(\mu_0,\mu_t) - W_p(\nu,\mu_t) = t - W_p(\nu, \mu_t).$$
For $0 \le t_1 \le t_2$,
$$[t_2 - W_p(\nu,\mu_{t_2})] - [t_1 - W_p(\nu,\mu_{t_1})]
= W_p(\mu_{t_2},\mu_{t_1}) + W_p(\nu,\mu_{t_2}) - W_p(\nu,\mu_{t_1}) \ge 0.$$
Thereby $t - W_p(\nu,\mu_t)$ is bounded and non-decreasing with respect to $t$.
\end{remark}

\begin{remark}
$b_\mu$ is a Lipschitz function since
\begin{equation}\label{Lipschitz}
|b_\mu(\nu^2) - b_\mu(\nu^1)| = \lim\limits_{t \to \infty} |W_p(\nu^2, \mu_t) - W_p(\nu^1,\mu_t)| \le W_p(\nu^1,\nu^2).
\end{equation}
\end{remark}

Busemann functions have the following fundamental properties.
For the conventional case, we refer to \cite{Peterson}.
\begin{proposition}\label{gradient line}
Let $(\mu_t)_{t \ge 0}$ be a unit-speed ray in $P_p(\mathcal{X})$.
If $(\nu_t)_{t \ge 0}$ is a co-ray from $\nu_0$ to $(\mu_t)_{t \ge 0}$, then
\begin{enumerate}[(i)]
\item $b_\mu(\nu_{t_1}) - b_\mu(\nu_{t_2}) = t_2 - t_1, \text{~for any~} t_1, t_2 \ge 0$;
\item $b_\mu(\lambda) \le b_\nu(\lambda) + b_\mu(\nu_0), \text{~for each~} \lambda \in P_p(\mathcal{X})$.
\end{enumerate}
\end{proposition}

\begin{proof}
(i) Assume the geodesic sequence $(\nu_t^n)_{0 \le t \le L_n} \in T(\nu_0^n,\mu_{t_n})$ converges to $(\nu_t)_{t \ge 0}$ as $n \to \infty$.
Let $0 \le t_1 \le t_2$, for any $n > t_2$ we have
$W_p(\nu_0^n,\mu_{t_n}) = W_p(\nu_0^n,\nu_{t_i}^n) + W_p(\nu^n_{t_i},\mu_{t_n})$, $i = 1,2.$
Then
\begin{eqnarray*}
& & |W_p(\nu_{t_1},\mu_{t_n}) - W_p(\nu_{t_2},\mu_{t_n}) - (t_2 - t_1)|\\
& = & |W_p(\nu_{t_1},\mu_{t_n}) - W_p(\nu_{t_2},\mu_{t_n}) - [W_p(\nu^n_{t_1},\mu_{t_n}) - W_p(\nu^n_{t_2},\mu_{t_n})]|\\
& \le & W_p(\nu^n_{t_1},\nu_{t_1}) + W_p(\nu^n_{t_2},\nu_{t_2}) \to 0.
\end{eqnarray*}

(ii) For $s,t \ge 0$, by the triangle inequality,
\begin{eqnarray*}
W_p(\lambda,\mu_s) - s & \le & W_p(\lambda,\nu_t) + W_p(\nu_t,\mu_s) - s\\
& = & [W_p(\lambda,\nu_t) - t] + [W_p(\nu_t,\mu_s) - s] + t.
\end{eqnarray*}
Letting $s \to +\infty$ and applying (i) we have
\begin{eqnarray*}
b_\mu(\lambda) & \le & [W_p(\lambda,\nu_t) - t] + [b_\mu(\nu_t) + t]\\
& = & [W_p(\lambda,\nu_t) - t] + b_\mu(\nu_0)
\end{eqnarray*}
Let $t \to +\infty$, then the inequality follows.
\end{proof}

For $p > 1$, $\mathcal{X}$ is non-branching if and only if $P_p(\mathcal{X})$ is non-branching (see \cite[Corollary 7.32]{Villani}).
Based on this fact, we have the following lemma.

\begin{lemma}\label{nonbranching coray}
Assume $\mathcal{X}$ is non-branching.
Let $(\mu_t)_{t \ge 0}$ be a unit-speed ray in $P_p(\mathcal{X})$.
If there is another unit-speed ray $(\nu_t)_{t \ge 0}$ satisfying
\begin{equation}\label{negative gradient line}
b_\mu(\nu_t) - b_\mu(\nu_s) = s - t, \text{~for any~} s,t \ge 0,
\end{equation}
then for any $\tau > 0$, the subray $(\nu_{t+\tau})_{t \ge 0}$ is the unique co-ray from $\nu_\tau$ to $(\mu_t)_{t \ge 0}$.
\end{lemma}

\begin{proof}
For $\tau > 0$, assume $(\tilde{\nu}_t)_{t \ge 0}$ is a co-ray from $\nu_\tau$ to $(\nu_t)_{t \ge 0}$.
Let
\begin{equation}
\nu'_t = \left\{
\begin{aligned}
&\nu_t, &0 \le t \le \tau,\\
&\tilde{\nu}_{t - \tau}, &t \ge \tau.
\end{aligned}\right.
\end{equation}
We claim that $(\nu'_t)_{t \ge 0}$ is a ray.
Since $(\nu_t)_{t \ge 0}$ and $(\tilde{\nu}_t)_{t \ge 0}$ are unit-speed rays, it is obvious for $s,t \le \tau$ or $s,t \ge \tau$ that
$$W_p(\nu'_s,\nu'_t) = |s - t|.$$
For $s < \tau < t$, by triangle inequality and the definition of rays,
\begin{eqnarray*}
W_p(\nu'_s,\nu'_t) & \le & W_p(\nu'_s,\nu'_\tau) + W_p(\nu'_\tau, \nu'_t)\\
& = & W_p(\nu_s,\nu_\tau) + W_p(\tilde{\nu}_0, \tilde{\nu}_{t - \tau})\\
& = & (\tau - s) + (t - \tau)\\
& = & t-s.
\end{eqnarray*}
On the other hand, by (\ref{negative gradient line}) and Proposition \ref{gradient line},
\begin{eqnarray*}
t - s & = & (\tau - s) + (t - \tau)\\
& = & [b_\mu(\nu_s) - b_\mu(\nu_\tau)] + [b_\mu(\tilde{\nu}_0) - b_\mu(\tilde{\nu}_{t - \tau})]\\
& = & [b_\mu(\nu'_s) - b_\mu(\nu'_\tau)] + [b_\mu(\nu'_\tau) -  b_\mu(\nu'_t)]\\
& = & b_\mu(\nu'_s) -  b_\mu(\nu'_t)\\
& \le & W_p(\nu'_s, \nu'_t).
\end{eqnarray*}
Here the last inequality follows from (\ref{Lipschitz}).
This proves our claim.
Since $P_p(\mathcal{X})$ is non-branching, the ray $(\nu'_t)_{t \ge 0}$ coincides with $(\nu_t)_{t \ge 0}$.
\end{proof}
Notice that (\ref{negative gradient line}) holds for every co-ray.
Lemma \ref{nonbranching coray} shows the second conclusion of Theorem \ref{main result 2}.

\begin{theorem}\label{gradient line op}
Assume $\mathcal{X}$ is non-branching.
Let $(\mu_t)_{t \ge 0}$ and $(\nu_t)_{t \ge 0}$ be two unit-speed rays in $P_p(\mathcal{X})$.
If there exists a sequence $\{t_n\}$ tending to $0$ such that $(\nu_{t+t_n})_{t \ge 0}$ is the unique co-ray from $\nu_{t_n}$ to $(\mu_t)_{t \ge 0}$,
then $(\nu_t)_{t \ge 0}$ is a co-ray from $\nu_0$ to $(\mu_t)_{t \ge 0}$.
\end{theorem}

\begin{proof}
Extract a decreasing subsequence of $\{t_n\}$, denote by $\{t_{n'}\}$, such that $t_{n'} \le 2^{-n'}$.
Let $\tau \ge 1$ be a constant.
For each $n'$, by Lemma \ref{nonbranching coray}, $(\nu_{t+t_{n'}})_{t \ge 0}$ is the unique co-ray from $\nu_{t_{n'}}$ to $(\mu_t)_{t \ge 0}$.
Then there exist $s_{n'} > \max\{\tau, n'\}$ and geodesic $(\nu^{n'}_t)_{0 \le t \le L_{n'}} \in T(\nu_{t_{n'}}, \mu_{s_{n'}})$ such that
$W_p(\nu^{n'}_\tau, \nu_{\tau + t_{n'}}) < 2^{-n'}$, which gives
\begin{equation}\label{coray construction}
W_p(\nu^{n'}_\tau,\nu_\tau) < 2^{-n'+1}
\end{equation}
where $L_{n'} = W_p(\nu_{t_{n'}},\mu_{s_{n'}})$.

By Theorem \ref{coray existence}, there exists a subsequence $\{(\nu^{n''}_t)_{0 \le t \le L_{n''}}\}$
converging to a co-ray $(\lambda_t)_{t \ge 0}$ from $\nu_0$ to $(\mu_t)_{t \ge 0}$.
It suffices to show that $(\lambda_t)_{t \ge 0}$ coincides with $(\nu_t)_{t \ge 0}$.
We can see from (\ref{coray construction}) that $\lambda_\tau = \nu_\tau$.
Let $\tilde{\nu}^{n''}_t = \nu^{n''}_{t + \tau}$, then $(\tilde{\nu}^{n''}_t)_{0 \le t \le L_{n''} - \tau}$ converges to the co-ray $(\lambda_{t + \tau})_{t \ge 0}$.
By Lemma \ref{nonbranching coray}, $\lambda_t = \nu_t$ for every $t \ge \tau$.
Thus the conclusion follows from that $P_p(\mathcal{X})$ is non-branching.
\end{proof}

Combining Proposition \ref{gradient line}, Lemma \ref{nonbranching coray} and Theorem \ref{gradient line op}, we obtain Theorem \ref{main result 3}.

\begin{definition}
Let $(\mathcal{Y}, d)$ be a metric space.
The set
$$\mathcal{V}(\mathcal{Y})= \left\{u: \mathcal{Y} \to \mathbb{R}\left| \text{~for any~} y \in \mathcal{Y}, u(y)= \min_{z \in \mathcal{Y} \backslash \{y\}}\{d(y,z)+u(z)\}\right.\right\}$$
is said to be the metric viscosity class of $\mathcal{Y}$.
\end{definition}
For $u\in \mathcal{V}$ we mean that:
\begin{itemize}
\item for any $x$ and $y$, $u(x)\leq d(x,y)+u(y)$;
\item for any $x$, there exists $y\neq x$ such that $u(x)=d(x,y)+u(y)$.
\end{itemize}
On a non-compact complete Riemannian manifold $(M, g)$, $u$ is a viscosity solution to the Eikonal equation $|\nabla u|_g = 1$ if and only if $u \in \mathcal{V}(M)$ (see for instance \cite{Cui, Fathi}).
Besides, Busemann functions on $(M,g)$ are viscosity solutions to the Eikonal equation (see e.g. \cite{CuiCheng}).
Unfortunately, for $P_p(\mathcal{X})$ with $p \neq 2$, viscosity solutions can not be defined as usual because lack of proper differential structure.
These facts motivate us to consider such a set for $P_p(\mathcal{X})$.

\begin{theorem}
If $(\mu_t)_{t \ge 0}$ is a unit-speed ray in $P_p(\mathcal{X})$, then $b_\mu \in \mathcal{V}(P_p(\mathcal{X}))$.
\end{theorem}

\begin{proof}
For any $\nu_0,\lambda \in P_p(\mathcal{X})$, from (\ref{Lipschitz}) we have $b_\mu(\nu_0) \le W_p(\nu_0, \lambda) + b_\mu(\lambda)$.
By Theorem \ref{coray existence}, there exists a co-ray $(\nu_t)_{t \ge 0}$ from $\nu_0$ to $(\mu_t)_{t \ge 0}$.
Thus by Proposition \ref{gradient line},
\begin{eqnarray*}
b_\mu(\nu_0) & = & t + b_\mu(\nu_t)\\
& = & W_p(\nu_0,\nu_t) + b_\mu(\nu_t)\\
& = & \min_{\lambda \in P_p(\mathcal{X}) \backslash \{\nu_0\}} W_p(\nu_0,\lambda) + b_\mu(\lambda).
\end{eqnarray*}
\end{proof}

\bibliographystyle{amsplain}

\end{document}